\documentclass[12pt,oneside,reqno]{amsart}       
\usepackage{enumerate}
\usepackage{tikz}
\usepackage{a4,amsmath,amssymb,amsthm,amscd,enumerate,amscd,}
\usepackage{tikz}
\usepackage{amsfonts,eufrak}
\usepackage[utf8]{inputenc}
\usepackage{amsmath}
\usepackage{amsfonts}
\usepackage{graphicx}
\usepackage{amssymb}
\usepackage{amsthm}
\usepackage[colorlinks=true, urlcolor=Blue, pdfborder={0 0 0}]{hyperref}

\usepackage{graphicx}
\usepackage{forest}
\usepackage{tikz-qtree}

\makeatletter
\@namedef{subjclassname@2020}{%
	\textup{2020} Mathematics Subject Classification}
\makeatother
\theoremstyle{definition}
\newtheorem{theorem}{Theorem}[section]
\newtheorem{lemma}[theorem]{Lemma}
\newtheorem{proposition}[theorem]{Proposition}

\theoremstyle{definition}
\theoremstyle{definition}

\newtheorem{remark}[theorem]{Remark}
\newtheorem{example}[theorem]{Example}
\usepackage{indentfirst}

\begin{document}
	\baselineskip=16.5pt
	\title[]{ involutions on the product of  projective space  and  sphere}
	\author[Dimpi and Hemant Kumar Singh]{ Dimpi and Hemant Kumar Singh}
	\address{ Dimpi \newline 
		\indent Department of Mathematics\indent \newline\indent University of Delhi\newline\indent 
		Delhi -- 110007, India.}
	\email{dimpipaul2@gmail.com}
	\address{  Hemant Kumar Singh\newline\indent 
		Department of Mathematics\newline\indent University of Delhi\newline\indent 
		Delhi -- 110007, India.}
	\email{hemantksingh@maths.du.ac.in}

	\date{}
	\thanks{The first author of the paper is  supported by SRF of UGC, New Delhi, with  reference no.: 201610039267.}
	\begin{abstract} 
		
		Let $G=\mathbb{Z}_2$ act on a  finite  CW-complex $X$ having mod $2$  cohomology isomorphic to the product of  projective space and sphere $\mathbb{F}P^n\times \mathbb{S}^m,$ where    $\mathbb{F}=\mathbb{R}$  or $\mathbb{C}.$ In this  paper, we have  determined the connected  fixed point sets and the orbit spaces of free involutions  on $X.$ As an application, we derive the Borsuk-Ulam type results.	
	\end{abstract}
	

	\maketitle
	\section {Introduction}
	\noindent	Let $G$ be a compact Lie group  acting on a finite CW-complex $X$ with the fixed point set $F.$ The study of  fixed point sets has attracted many mathematicians over the world, since  the beginning of the twentieth century. Smith \cite{s} proved that the fixed point sets of $G=\mathbb{Z}_p,$ $p$ a prime, on a finite  dimensional polyhedron $X$ having mod $p$ cohomology $n$-sphere (resp. $n$-disc) are  mod $p$  cohomology $r$-sphere (resp. $r$-disc). He has also proved that if $G=\mathbb{Z}_2$ acts effectively on the real projective space then the fixed point set is either empty or it has two components having mod $2$ cohomology real projective space \cite{smith}.  Bredon \cite{ Bredon} generalizes this result for $G=\mathbb{Z}_p,~ p$ a prime, actions on  cohomology projective spaces. Su also proved  similar results for (i) $G=\mathbb{Z}_p,$ $p$ odd prime actions on mod $p$ cohomology lens spaces, and $G= \mathbb{S}^1$ actions on cohomology complex projective spaces \cite{j2}, (ii) $G=\mathbb{Z}_p, p$ a prime, actions on a space $X$ with mod $p$ cohomology  product of spheres $\mathbb{S}^m \times \mathbb{S}^n$ \cite{j1}.   Bredon \cite{Bredon} proved a more general result. He proved  that if a finitistic space $X$  satisfies poincar\'{e} duality with respect to  \v{C}ech cohomology with $\mathbb{Z}_p$-coefficients then each component of the fixed point set also satisfies poincar\'{e}  duality.  Puppe \cite{puppe} proved Bredon's  conjecture. He has  shown that  if $X$ is totally nonhomologous to zero in $X_G$ (Borel space) then the number of generators of the cohomology ring of each component of the fixed point set with $\mathbb{Z}_p$-coefficient  is  at most the number of generators of $H^*(X).$ Following this thread of research,  the fixed point sets of involution on the product of projective spaces and Dold manifold have been determined in \cite{chang, peltier}. \\
	\indent On the other hand, if $F$ is empty, then the study of orbit spaces is also an interesting problem.
	It is well known that the orbit spaces of free actions of  $\mathbb{Z}_2,~ \mathbb{S}^1$ and $\mathbb{S}^3$ on  $\mathbb{S}^n,$ $\mathbb{S}^{2n+1}$ and $\mathbb{S}^{4n+3}$ are projective spaces $P^n(q),$ where $q=1,2$ and $4,$ respectively. Recently, the orbit spaces of free involutions on real Milnor manifolds, Dold manifolds, and the product of two projective spaces have been discussed in \cite{dey,Mattos,msingh}. Further, the orbit spaces of free actions of $G=\mathbb{Z}_p, p$ a prime, or $G=\mathbb{S}^d, d=1$ or $3,$ on the cohomology product of spheres $\mathbb{S}^n \times \mathbb{S}^m$ have been studied in \cite{orbit space, anju}.

	In continuation, it seems to be an interesting problem to determine the possibilities of the fixed point sets  and the orbit spaces of free involutions on a space $X$ with mod $2$ cohomology isomorphic to the  product of  projective space and  spheres $ \mathbb{F}P^n \times \mathbb{S}^m,$ where $\mathbb{F}=\mathbb{R}$  or $\mathbb{C}.$ In this paper, we have determined the possibilities of the connected fixed point sets of involution on $X$ and classified  the  orbit spaces of free involutions on $X.$ 
	\section{Preliminaries}  
	\noindent In this section, we recall  some known facts that will be used in this paper.  Let $G$ be a finite cyclic group  acting on a finite CW-complex  $X.$ The associated Borel fibration is $ X \stackrel{i} \hookrightarrow X_G \stackrel{\pi} \rightarrow B_G,$ where $X_G = (X\times E_G)/G$ is  Borel space obtained by diagonal action of $G$ on space $X\times E_G$ and $B_G$ is classifying space  of $G$ which is the orbit space of a  free action of $G$ on contractible space $E_G.$  Suppose that the fixed point set of an  action $G$ on $X$ is nonempty. Let $x \in F$ and $\eta_x : B_G \hookrightarrow X_G$ be a cross section of projection map $\pi: X_G \rightarrow B_G ,$ where $B_G \approx (\{x\} \times E_G)/G ,$ then $H^*(X_G)\cong
	ker~ \eta^*_x \oplus Im~ \pi^*$ \cite{Bredon}. A space $X$ is said to be  totally nonhomologous to zero (TNHZ)
	in $X_G$ if the inclusion map $i: X \hookrightarrow X_G$ induces
	a surjection in the cohomology $ i^*: H^*(X_ G) \rightarrow H^*(X). $ We have used the following Propositions to prove our main results related to fixed point sets.

	\begin{proposition}(\cite{bredon})
		Let $G=\mathbb{Z}_2$  act on  a finite CW-complex $X$ and $\sum$ rk $H^i(X,\mathbb{Z}_2)< \infty.$ Then, the following statements are equivalent: \\
		(a) $X$ is TNHZ (mod 2) in $X_G.$   \\
		(b) $\sum$ rk $H^i(F,\mathbb{Z}_2) =\sum$ rk $H^i(X,\mathbb{Z}_2).$ \\ 
		(c) $G$ acts trivially on $H^*(X; \mathbb{Z}_2)$ and spectral sequence $E^{r,q}_2$ of $X_G \rightarrow B_G$ degenerates. 
	\end{proposition} 
	\begin{proposition}(\cite{Bredon})\label{2.2}
		Let $X$ be TNHZ in $X_G$ and $\{\gamma_j\}$ be a set of homogeneous elements in $H^*(X_G;\mathbb{Z}_p)$ such that $\{i^*(\gamma_j)\}$ forms $\mathbb{Z}_p$-basis of $H^*(X; \mathbb{Z}_p).$ Then, $H^*(X_G;\mathbb{Z}_p)$ is the free $H^*(B_G)$-module generated by $\{\gamma_j\}.$
	\end{proposition}
	\begin{proposition}\label{thm 2}(\cite{bredon})
		Let $G = \mathbb{Z}_2$ act on the finite CW-complex
		$X$ and  $A \subset X$ be closed and invariant subspace. Suppose that $H^i ( X , A ; \mathbb{Z}_2)
		= 0$ for $i > n.$ Then, the homomorphism       
		$$j ^*: H^ k ( X_G ,A_G; \mathbb{Z}_2) \rightarrow H^k(F_G, F_G \cap A_G; \mathbb{Z}_2)$$
		is an isomorphism for $k > n$. If $( X , A )$ is TNHZ (mod $2$) in $( X_ G , A_G,)$ then $j^*$ is a monomorphism for all $k.$ 
	\end{proposition}
	Note that \cite{G.Bredon} the induced homomorphism $\eta_x^*$ depends on the component $F_0$ of fixed point set $F$ in which $x$ lies. If $\alpha \in H^n(X_G)$ such that $\alpha \in Ker~ \eta_x^*$ then the restriction of $j: (F_G,x_G) \hookrightarrow (X_G,x_G)$ on $(F_0)_G$ denoted by $j_0$ and the term of  $j_0^*(\alpha)$ involving $H^0(F_0,x_G)$ vanishes.

	\begin{proposition}\label{prop 2.4}(\cite{bredon})
		Let $G=\mathbb{Z}_2$ act on a finite CW-complex $X$ and $X$ is TNHZ in $X_G.$ Then, $a|F$ is nontrivial element of the fixed point set $F,$ for any class $a\in H^n(X;\mathbb{Z}_2)$ such that $a^2 \neq 0.$
	\end{proposition}	
	
	Next, we recall some results of Leray-Serre spectral sequence associated to Borel fibration $ X \stackrel{i} \hookrightarrow X_G \stackrel{\pi} \rightarrow B_G.$ For the proof we refer \cite{bredon,mac}
	\begin{proposition}\label{prop 5}(\cite{bredon})
		Let $X$ be a finite CW-complex $G$-space, where $G=\mathbb{Z}_2$ and $g$ be a generator of $G.$ Suppose that $H^i(X; \mathbb{Z}_2)=0$ for all $i>2n,$ and $H^{2n}(X;\mathbb{Z}_2)=\mathbb{Z}_2.$ If an element $a \in H^{n}(X;\mathbb{Z}_2)$ such that $ag^*(a)\neq 0,$ then fixed point set is non empty. Moreover, the element $ag^*(a) \in E_2^{0,2n}$ is a  permanent cocycle in the Leray-Serre  spectral sequence of fibration $X \hookrightarrow X_G \rightarrow B_G,$ for any $a \in H^{n}(X; \mathbb{Z}_2).$
	\end{proposition}
	\begin{proposition}\label{E_2} (\cite{bredon})
		Let $G=\mathbb{Z}_2$ act on a finite CW-complex $X$ and $g$ be the generator of $G$. Then, the $E_2$ term of the Leray-Serre spectral sequence of the fibration  $ X \stackrel{i} \hookrightarrow X_G \stackrel{\pi} \rightarrow B_G,$ is given by
		\[
		E_2^{k,i}=
		\begin{cases}
			\text{ker $\tau$} & \mbox{for}~ ~k=0 \\
			\text{ker $\tau$/Im $\sigma$} & \text{for } k>0\\

		\end{cases}
		\]
		where $\tau = \sigma = 1+g^*.$
	\end{proposition}
	\begin{proposition}(\cite{mac})
		Suppose that  the system of local coefficients on $B_G$ is simple.	Then, the  homomorphisms $i^*: H^*(X_G) \rightarrow H^*(X)$ and $\pi^*: H^*(B_G) \rightarrow H^*(X_G)$ are the edge homomorphisms, \begin{center}
			$ H^k(B_G)=E_2^{k,0}\rightarrow E_3^{k,0}\rightarrow \cdots E_k^{k,0}\rightarrow E_{k+1}^{k,0} = E_{\infty}^{k,0} \subset H^k(X_G)$ and $  H^i(X_G) \rightarrow E_{\infty}^{0,i} \hookrightarrow E_{l+1}^{0,i} \hookrightarrow E_{l}^{0,i} \hookrightarrow \cdots \hookrightarrow E_{2}^{0,i} \hookrightarrow E_2^{0,i} \hookrightarrow H^i(X),$ respectively.
		\end{center} 
	\end{proposition}
	\begin{proposition}(\cite{mac})\label{mac}
		Let $G=\mathbb{Z}_2$   act freely on a finite CW-complex $X.$ Then, the Borel space $X_G$ is homotopy equivalent to the orbit space $X/G.$
	\end{proposition}
	\begin{proposition}(\cite{bredon})\label{prop 4.5}
		Let $G=\mathbb{Z}_2$   act freely on a finite CW-complex $X.$ If $H^i(X;\mathbb{Z}_2)=0~ \forall~ i>n,$ then $H^i(X/G;\mathbb{Z}_2)=0~ \forall~ i>n.$ 
	\end{proposition}
	\noindent Recall  that 
	$H^*(\mathbb{F}P^n \times \mathbb{S}^m; \mathbb{Z}_2)=\mathbb{Z}_2[a,b]/<{a^{n+1},b^2}>,$ where deg $b=m$ and deg $a=l, l=1$ or $2$   for $\mathbb{F}= \mathbb{R}$ or $ \mathbb{C},$  respectively.\\
	\noindent Throughout the paper,  
	$H^*(X)$ will denote the \v{C}ech cohomology of a space $X,$ and   $X\sim_2 Y,$  means $H^*(X;\mathbb{Z}_2 )\cong H^*(Y;\mathbb{Z}_2).$

	\section{Fixed Point Sets of  a product of projective space and  Sphere}                               
	\noindent	Let $G=\mathbb{Z}_2$ act on  a finite CW-complex $X \sim_{2} \mathbb{F}P^n \times\mathbb{S}^m;$ $ n, m \geq 1,$  where $F= \mathbb{R}$ or $\mathbb{C}.$ Suppose that the fixed point set $F$ is nonempty and $X$ is TNHZ in $X_G.$ 
	In this section, we have determined the possibilities of connected fixed point sets of involutions on $X.$ 
	
	\noindent First, we prove the following result:
	\begin{theorem}\label{thm 3.1}
		Let $G=\mathbb{Z}_2$ act on a finite CW-complex $X \sim_{2} \mathbb{R}P^n \times\mathbb{S}^m,$ $ n,m\geq 1.$  If $X$ is TNHZ in $X_{G},$ then the connected fixed point sets are one of the following:   
		\begin{enumerate}
			\item $F \sim_{2} \mathbb{R}P^n \times\mathbb{S}^q,$ where $1 \leq q \leq m,$ and
			\item  $F\sim_2 \mathbb{R}P^{n+1}\# \mathbb{R}P^{n+1}.$
		\end{enumerate}
	\end{theorem} 
	\begin{proof}
		Let $x \in  F$ and  $\{a,\cdots, a^n,b ,ab,\cdots a^nb\}$ be a  generating set of $H^*(X,x),$ where deg $a=1$ and deg $b=m.$ Since $X$ is TNHZ in $X_G,$ we get  rk $H^*(F)=2n+2,$ $\pi_1(B_G)$ acts trivially on $H^*(X_G,x_G)$ and the $E_2$-term $E_2^{p,q}=H^p(B_G) \otimes H^q(X)$ of Leray-Serre spectral sequence of the Borel fibration $ X \stackrel{i} \hookrightarrow X_G \stackrel{\pi} \rightarrow B_G$  is $E_{\infty}^{p,q}.$ So, the elements $\{1 \otimes a,1 \otimes a^2,\cdots, 1 \otimes a^n,1 \otimes b ,1 \otimes ab,\cdots 1 \otimes a^nb\}$ are permanent cocycles. Assume that $\alpha \in H^1(X_G,x_G)$ represents  generator $a \in H^1(X,x)$ and $\beta \in H^m(X_G,x_G)$ represents  generator $b \in H^m(X,x)$ such that $\eta_x^*(\alpha)=\eta_x^*(\beta)=0,$ where $\eta_x: \frac{\{x\} \times E_G}{G} \hookrightarrow \frac{X \times E_G}{G}$ is the inclusion map.  By Proposition \ref{2.2}, $\{\alpha, \alpha^2, \alpha^3, \cdots \alpha^n, \beta ,\alpha \beta , \alpha ^2 \beta, \cdots \alpha^n \beta\}$ is a generating set of $H^*(X_G,x_G)$ over $H^*(B_G)$-module.
		As the homomorphisms $j^* : H^k(X_G,x_G) \rightarrow H^k(F_G,x_G)$ are injective for all $k,$ we get $j^*(\alpha)= 1 \otimes c,$ where $0\neq c\in H^1(F,x).$ We know that $i_1^*j^*=j_1^*i^*$ where $i_1: F \hookrightarrow F_G$ and $ j_1: F \hookrightarrow X$  are the inclusion maps. So,  we get  $a|_{F} =c.$  Consequently, we get $j^*(\alpha^i)= 1 \otimes c^i, 2 \leq i \leq n.$ Obviously, $c^{n+1}=0.$ We have 	$ H^m(F_G,x_G)= \bigoplus_{i=0}^{m} H^{m-i}(B_G) \otimes H^i(F,x)$ and $\eta_x(\beta)=0.$  We may assume that	 
		\begin{equation*}
			j^*(\beta)=1 \otimes d_m +  t \otimes d_{m-1} + t^2 \otimes d_{m-2} \cdots + t^k \otimes d_{m-k} \cdots + t^{m-1}\otimes d_{1}.
		\end{equation*}
		Since $j^*$ is injective, the images of generating set of $H^*(X_G,x_G)$ are linearly independent. This implies that for some $i,$ $d_i \neq c^i, 1 \leq i \leq n,$ where deg $d_i=i.$ Suppose $d_q=d_{m-k}=d$ is the least degree element which is not of the form $c^i.$ As $j^*$ is onto on high degrees, for sufficiently large value of $r,$ we can write
		\begin{align*}
			t^{k+r} \otimes d=j^*(A_1 t^{r+m-1}\alpha  +\cdots + A_n t^{r+m-n}\alpha^n+A_m t^{r}\beta + \cdots+A_{m+n} t^{r-n}\alpha^{n} \beta),
		\end{align*}
		where $A_i's$ are in $\mathbb{Z}_2.$ After comparing the coefficient of $t^{k+r}\otimes d,$ we get $A_m=1.$ So, we have 
		\begin{equation*}
			t^{k+r} \otimes d - j^*(t^r\beta)= j^*(A_1t^{r+m-1}\alpha +  \cdots +A_n t^{r+m-n}\alpha^n+ A_{m+1}t^{r-1}\alpha \beta +\cdots+ A_{m+n}t^{r-n}\alpha^n \beta)
		\end{equation*}
		It is easy to observe that
		$$
		t^{r} \otimes d_m+\cdots +t^{r+k-1} \otimes d_{m-(k-1)}+t^{r+k+1} \otimes d_{m-(k+1)}+ \cdots + t^{r+m-1} \otimes d_1$$
		$$	=-j^*(A_1 t^{r+m-1}\alpha + \cdots + A_2 t^{r+m-n}\alpha^n +A_m t^{r}\beta+ \cdots  + \cdots+A_{m+n} t^{r-n}\alpha^{n} \beta).
		$$
		From the above equation, if $q\leq n,$ then $d_i=c^i$ for $ 1\leq i \leq q-1,$  $ d_{i}= c^{i}+c^{i-q}d$ for  $q+1 \leq i \leq n~ \&~$  $d_i=c^{i-q}d$ for $n+1\leq i \leq n+q,$ and zero otherwise. And if $q>n,$ then $d_i=c^i$ for $ 1\leq i \leq n~\&~$   $d_i=c^{i-q}d$ for $q+1 \leq i \leq n+q,$ and zero otherwise. Note that $d_i=0~ \forall~i>m.$  Since, $c^{n+1}=0,$ we have $j^*(\alpha^n\beta )=t^r\otimes c^nd.$ As $\alpha^n\beta \neq 0, $ we get  $c^nd \neq 0$ and hence, $c^id \neq 0$ for $ 1 \leq i \leq n.$ Clearly, $c^id\neq c^{i+q},$ for $q < n$ and  $1 \leq i \leq n-q.$ As $H^*(F)=2n+2,$ $\{1,c,c^2,\cdots c^n,d,cd,\cdots c^nd\}$ forms a basis of $H^*(F).$ Clearly, if $d^2=0$ then  $F \sim_2 \mathbb{R}P^n \times \mathbb{S}^q, 1 \leq q \leq m.$ This realizes possibility (1). And if $d^2 \neq 0,$ then  $d^2=c^{2q}$ or $d^2=c^qd.$ If $d^2=c^{2q},$ then we must have $2q\leq n$ and by the change of basis $d'=d+c^q,$ we get $d'^2=0 $ and $c^id' \neq 0$ for $1 \leq i \leq n.$ Thus, again we get possibility (1). Now, Let  $d^2=c^{q}d.$ It is easy to observe that if $1<q\leq n,$ then $F$ does not satisfy poincar\'{e} duality. So, we must have $q=1.$ Again, by the change of basis $d'=c+d,$  we get $d'^{n+2}=d^{n+2}=d'^{n+1}+ d'^{n+1} = dd'=0.$ Thus, $F\sim_2 \mathbb{R}P^{n+1}\#\mathbb{R}P^{n+1},$ the connected sum of real projective spaces of height $n+1.$ This realizes possibility (2). 
	\end{proof} 
	\begin{remark}
		For $n=1,$ the mod $2$ cohomology of $X$ is isomorphic to $\mathbb{S}^1 \times \mathbb{S}^m.$ By Theorem \ref{thm 3.1}, it is clear that the  possible connected  fixed point sets are either  $F\sim_2 \mathbb{S}^1 \times \mathbb{S}^q, 1 \leq q \leq m$ or $F\sim_2 \mathbb{R}P^2 \# \mathbb{R}P^2.$ These possibilities have also been realized in [Theorem 3.11, \cite{j1}].
	\end{remark}
	\noindent Next, we have  determined the possibilities of connected fixed point sets of involutions on $X\sim_{2} \mathbb{C}P^n \times \mathbb{S}^m,n, m\geq 1.$
	\begin{theorem}\label{thm 3.2}
		Let $G=\mathbb{Z}_2$ act on a finite CW-complex $X \sim_{2} \mathbb{C}P^n \times\mathbb{S}^m,$ $ n,m \geq 1.$  If $X$ is TNHZ in $X_{G},$ then the connected fixed point sets are one of the following:  
		\begin{enumerate}
			\item $F \sim_{2} \mathbb{F}P^n \times\mathbb{S}^q,$ where $\mathbb{F}=\mathbb{R}$ or $\mathbb{C},$ $1 \leq q \leq m.$
			\item   $F\sim_2 \mathbb{F}P^{n+1}\# \mathbb{F}P^{n+1},$ where  $\mathbb{F}=\mathbb{R}$ or $\mathbb{C}.$			
			\item $ H^*(F) $ is generated by $c$ and $d,$ with  $c^{n+1}=d^2+c^{q}=d^{2l+2}=0,$ where deg $c=2,$ deg $d=q$  and $l=[\frac{n}{q}],$ $ ~q$ odd.	Moreover, for $q=1,$ $F \sim_{2} \mathbb{R}P^{2n+1}.$
			\item $ H^*(F) $ is generated by $c$ and $d,$ with  $c^{r+1}=c^r+d^{\frac{r}{q}}=cd=0,$ where deg $c=1,$ deg $d=q\leq n,$ $r=q(2n+2)/(q+1),$ and $n=(q+1)k-1,$ for $q$ even $\&$ $n=(\frac{q+1}{2})k-1,$ for $q$ odd. Moreover, for $q=1,$ $F \sim_2 \mathbb{R}P^{k}\# \mathbb{R}P^{k},$ $n=k-1$ and  for $q=2,$ $F \sim_2 \mathbb{R}P^{4k}\# \mathbb{C}P^{2k},$  $n=3k-1$ for some $ k \in \mathbb{N}.$

			\item $ H^*(F) $ is generated by $c$ and $d,$ with   $c^{r+1}=c^{qj}+d^{j}=c^{r-q+1}d=0,$ where deg $c=1,$ deg $d=q< n,$   $r=\frac{2n+2}{j}+qj-(q+1),$  $n+1= jk,$ for some $k\in \mathbb{N},$ or $j=2,$  and $n>\frac{(q+1)j}{2}-1.$

		\end{enumerate}
	\end{theorem} 
	\begin{proof}
		Let $x \in  F$ and   $\{a,\cdots, a^n,b, ab,\cdots a^nb\}$  be a  generating set of $H^*(X,x),$ where deg $a=2$ and deg $b=m.$  As in proof of Theorem \ref{thm 3.1},  we choose $\alpha \in H^2(X_G,x_G)$ and $\beta \in H^m(X_G,x_G)$  representing the  generators $a \in H^2(X,x)$ and $b \in H^m(X,x)$ with $\eta_x^*(\alpha)=\eta_x^*(\beta)=0,$ respectively. Let $j^*(\alpha)= B_1 t \otimes c_1 + B_2 1 \otimes c_2,$ where $c_1\in H^1(F,x),$ $c_2 \in H^2 (F,x)$ and $B_1,B_2 \in  \mathbb{Z}_2.$ It is clear that $c_2=a|F.$ Also, we may write 
		\begin{equation*}
			j^*(\beta)=1 \otimes d_m +  t \otimes d_{m-1} + \cdots + t^k \otimes d_{m-k} \cdots + t^{m-2}\otimes d_{2}+ t^{m-1}\otimes d_1,
		\end{equation*} where $d_i \in H^i(F,x).$\\
		If $c_2\neq 0,$   then $B_2=1.$ Clearly, $c_2^{n+1}=0.$  So, we  consider two cases: (1)  $j^*(\alpha)=1 \otimes c_2,$ and (2) $j^*(\alpha)=  t \otimes c_1 + 1 \otimes c_2.$\\
		\textbf{Case (1):} When $j^*(\alpha)=1 \otimes c_2.$\\ In this case, $j^*(\alpha^i)= 1 \otimes c_2^i, 2 \leq i \leq n.$    As $j^*$ is injective, $d_j \neq c_2^j,$ for some $j,$ $1 \leq j \leq n,$ where deg $d_j=j.$ Suppose $d_q=d_{m-k}=d$ is the least degree element such that $d_q \neq c_2^j.$ 
		As $j^*$ is onto on high degrees, for sufficiently large value of $r,$ we can write
		\begin{align*}
			t^{k+r} \otimes d=j^*(A_1 t^{r+m-2}\alpha  +\cdots +A_n t^{r+m-2n}\alpha^n+ A_m t^{r}\beta + \cdots+A_{m+n} t^{r-2n}\alpha^{n} \beta),
		\end{align*}
		where $A_i's$ are in $\mathbb{Z}_2.$ After comparing the coefficient of $t^{k+r}\otimes d,$ we get $A_m=1.$ So, we have 
		$$
		t^{r} \otimes d_m+\cdots +t^{r+k-1} \otimes d_{m-(k-1)}+t^{r+k+1} \otimes d_{m-(k+1)}+ \cdots + t^{r+m-1} \otimes d_1$$
		$$	=-j^*(A_1 t^{r+m-2}\alpha + \cdots + A_n t^{r+m-2n}\alpha^n + A_{m+1} t^{r-2}\alpha \beta  + \cdots+A_{m+n} t^{r-2n}\alpha^{n} \beta)
		$$
		From the above equation, we get that if $q$ is even, then $d_i's$ are zero for $i$ odd. Also, we get  $d_{2i}=c_2^i, 1 \leq i \leq \frac{m-k-2}{2}$  and $ d_{m-(k-l)}= c_2^{l/2}d,$ where $l$ is even, and $2 \leq l \leq k \leq 2n.$ If $q$ is odd then $d_{2i}=c_2^{i}, $  $d_{2i+q}=c_2^{i}d$ where $1\leq i \leq n,$ and zero otherwise. Thus, we get $j^*(\alpha^n\beta )=t^k\otimes c_2^nd.$ As $\alpha^n\beta \neq 0, $ we get  $c_2^nd \neq 0,$ and hence $c_2^id \neq 0$ for $ 1 \leq i \leq n.$ Clearly, if $d^2=0$ then  $F \sim_2 \mathbb{C}P^n \times \mathbb{S}^q, 1 \leq q \leq m.$ This realizes possibility (1) for $\mathbb{F}=\mathbb{C}$. And, if $d^2 \neq 0,$ then we have either $d^2=c_2^{q}$ or $d^2=c_2^
		{q/2}d.$ 
		
		If $d^2=c_2^{q},$ then we get $q\leq n.$  If $q$ is even then by the change of basis $d'=d+c^{q/2},$ we get $d'^2=0 $ and $c_2^id' \neq 0$ for $1 \leq i \leq n.$ Thus, again we get possibility (1) for $\mathbb{F}=\mathbb{C}.$ If $q$ is odd, then $d^{2l}=c_2^{lq}, d^{2l+1}=c_2^{lq}d,$ where $1 \leq  l \leq [\frac{n}{q}].$ Clearly, $d^{2l+2}=0.$ So, we have $c_2^{n+1}=d^{2l+2}=d^2+c_2^q=0,$ where $l=[\frac{n}{q}].$ In particular, if $q=1$ then $F\sim_2 \mathbb{R}P^{2n+1}.$ This realizes possibility (2).
		
		If $d^2=c_2^{q/2}d,$ then $q$ must be $2.$ For that, if $2<q(\mbox{even}) \leq n,$ then $F$ does not satisfy poinca\'{r}e duality.  By the change of basis $d'=d+c_2,$ we get $d'^{n+2}=d^{n+2}=d'^{n+1}+d^{n+1}=dd'=0 .$ Thus, $F\sim_2 \mathbb{C}P^{n+1}\# \mathbb{C}P^{n+1}.$  This realizes possibility (2) for $\mathbb{F}=
		\mathbb{C}.$\\ \textbf{Case (2):} When $j^*(\alpha)=  t \otimes c_1 +  1 \otimes c_2.$ \\
		In this case, we consider two subcases (i) $c_1^2\neq c_2,$ and (ii) $c_1^2 = c_2.$ \\
		\noindent \textbf{Subcase (i):} Assume that $c_1^2\neq c_2.$ Note that  $H^*(F)$  is generated by $c_1$ and $c_2$ \cite{puppe}.
		Clearly, $d_1=c_1$ in $j^*(\beta).$ As $j^*$ is onto in high degrees, for sufficiently large value of $r,$ we get 
		\begin{align*}
			t^{r+(m-1)} \otimes c_1=j^*(A_1 t^{r+m-2}\alpha  +\cdots +A_n t^{r+m-2n}\alpha^n+ A_m t^{r}\beta + \cdots+A_{m+n} t^{r-2n}\alpha^{n} \beta),
		\end{align*}
		where $A_i's$ are in $\mathbb{Z}_2.$ After comparing the coefficient of $t^{r+m-1}\otimes c_1,$ we get $A_m=1.$ Consequently, $$t^{r} \otimes d_m + \cdots  t^{r+m-2} \otimes d_2
		=j^*(\sum_{i=1}^{n}A_i t^{r+m-2i}\alpha^i + \cdots   +\sum_{i=1}^{n} A_{m+i} t^{r-2i}\alpha^i \beta ).$$
		Now, we have divided this subcase according to $c_1^2=0$ or $c_1^2\neq 0.$\\
		First, we consider  $c_1^2=0.$ In this case, we get 
		\[
		j^*(\alpha^n)=
		\begin{cases}
			1 \otimes c_2^n & \mbox{if $n$ is even} \\
			1\otimes c_2^n + t \otimes c_2^{n-1}c_1 & \text{if $n$ is odd}. \\
		\end{cases}
		\]
		As $j^*(\alpha^n) \neq 0,$ $c_2^n\neq 0,$ for $n$ even. Further, if  $n$ is  odd then $c_2^n\neq 0,$ otherwise, rk$H^*(F)=2n<2n+2,$ which contradicts that $X$ is TNHZ in $X_G.$ This gives  that $d_{2k}=c_2^k,$ $d_{2k+1}=c_2^k c_1,$ where $1 \leq k \leq n.$ As $c_2^{n+1}=0,$  we get
		
		\[
		j^*(\alpha^n \beta)=
		\begin{cases}
			t^{m-1}\otimes c_2^n c_1 & \mbox{if $n$ is even} \\
			0 & \text{if $n$ is odd}. \\
		\end{cases}
		\]
		Thus, $n$ must be even. Clearly, $\{1,c_1,c_2\cdots ,c_2^n, c_2c_1,\cdots c_2^{n}c_1\}$ forms a generating set of $H^*(F).$ So, $F \sim_2 \mathbb{C}P^n \times \mathbb{S}^1.$ This realizes possibility (1) for $\mathbb{F}= \mathbb{C}$ and $q=1,$ when $n$ is even. \\
		Next, we consider $c_1^2\neq 0.$
		
		First, we assume that  $c_2^2=0.$ Then, $j^*(\alpha^n)=t^n \otimes c_1^n$ if $n$ is even, and $j^*(\alpha^n)=t^{n-1}\otimes c_1^{n-1}c_2+t^n \otimes c_1^n$ if $n$ is odd. We must have $c_1^n\neq 0.$ If the cup product $c_1c_2=0,$ then $F$ does not satisfy poincar\'{e} duality. So, $c_1c_2\neq 0.$ If $c_1^{n+1}\neq 0,$ then rk $H^*(F)>2n+2,$ a contradiction. So, we must have  $c_1^{n+1}=0.$ It is easy to observe that $ d_k=c_1^k +c_1^{k-2}c_2,$ where $2 \leq k \leq n+2.$ Thus, $j^*(\alpha^n\beta)=t^{m-2}\otimes c_1^nc_2 ,$ if $n$ is even, and $j^*(\alpha^n \beta )=0$ if $n$ is odd. So, $n$ must be even and $c_1^nc_2 \neq 0.$ Thus, $F\sim_2 \mathbb{R}P^n \times \mathbb{S}^2.$ This realizes possibility (1) for  $\mathbb{F}=\mathbb{R}$ and $q=2,$ when $n$ is even.
		
		Now, we consider $c_2^2\neq 0,$  we get $j^*(\alpha ^n)= \sum_{r=0}^{n} {{n}\choose{r}} t^r \otimes c_2^{n-r} c_1^r,$ and $d_{m-i}= \bigoplus_{2l+j=m-i,l,j\geq 0}A_{l,j}c_2^lc_1^j, l=j\ne 0 ~\&~A_{l,j} \in \mathbb{Z}_2.$ So, we have $$j^*(\alpha^n \beta )= \sum_{r=0}^{n}\sum_{i=0}^{m-1}{{n}\choose{r}} t^{r+i} \otimes d_{m-i}c_2^{n-r}c_1^r.$$
		In the above expression, it is not difficult to observe that for $0\leq r <n-1,$ the lowest degree terms in $c_1$ and $c_2$ are $c_2^{n-r}c_1^{r+1}. $ As $F$ satisfies poincar\'{e} duality, $c_2^{n-r}c_1^{r+1}\neq 0 \implies$ rk $H^i(F)>2n+2,$ a contradiction. We also note that if $c_2c_1^{n+1}\neq 0 $ or $c_2^2c_1^{n}\neq 0,$ then rk $H^*(F)>2n+2,$ a contradiction. Thus, $c_2^{n-r}c_1^{r+1}= 0 $ for $0\leq r <n-1,$  $c_2c_1^{n+1}=0 $ and $c_2^2c_1^{n}= 0.$   So, we get $$j^*(\alpha ^n \beta )= t^{n+m-2}\otimes c_2c_1^n+ \sum_{j=1}^{m}t^{n+m-j}\otimes c_1^{n+j}.$$
		As $j^*$ is injective, at least one of $c_1^nc_2$ and $c_1^{n+1}$ must be  nonzero. So,
		if the cup product $c_1c_2=0$ then $c_1^{n+1}$ must be non zero. Let $r$ be the formal dimension of $F.$ Then, $c_1^r \neq 0.$   If $c_1^{2i}= c_2^i$ for $2i<r$ then $ c_2^ic_1=c_1^{2i+1}\neq 0,$ a contradiction. Thus, $ c_1^{2i}\neq c_2^i$ for $2i<r.$ Hence, we must have  $c_1^{r}=c_2^{r/2}$ and $r$ is even. As rk $H^*(F)=2n+2,$ we get  $r+r/2=2n+2 \hspace{-.2cm}\implies \hspace{-.2cm}3|(n+1).$ So, we get   $n=3k-1, k \in \mathbb{N}.$ Thus,  $c_1^{4k+1}=c_2^{2k+1}=c_1^{4k}+c_2^{2k}=c_1c_2=0.$ This realizes possibility (4) for $q=2.$  
		
		Now, if the cup product $c_1c_2 \neq 0,$ then we consider possibilities according as  $ c_1^{n+1}= 0$ or      $c_1^{n+1} \neq 0.$

		Let $c_1^{n+1}= 0.$ Then, we must have $c_1^nc_2\neq 0.$ This case holds  only for $n>3.$ It is easy to observe that $c_2^2=c_1^4$ and $c_1^i \neq c_1^{i-2}c_2, 3 \leq i \leq n.$ So, we get $H^i(F)$ is generated by $c_1^i$ and $c_1^{i-2}c_2$ for $2 \leq i \leq n,$ and $H^{n+1}(F)$ and $H^{n+2}(F)$ are generated by $c_1^{n-1}c_2$ and $c_1^nc_2,$ respectively. Thus, the cohomology ring $H^*(F)$ is generated by $c_1$ and $c_2$ with $c_1^{n+1}=c_2^2+c_1^4=0.$ By the change of basis $d'=c_1^2+c_2,$ we get $d'^2=0$ and $c_1^nd' \neq 0.$ This realizes possibility (1) for $\mathbb{F}=\mathbb{R}$ and $q=2.$

		Let $ c_1^{n+1} \neq 0.$ Then, either $c_1^nc_2\neq 0$ or $c_1^nc_2= 0.$ Let $r$ be the formal dimension of $F.$ Then, we must have $c_1^r\neq 0.$ And if $c_1^i = c_1^{i-2}c_2$ for any  $3 \leq i \leq r,$ then $c_1^{r-2}$ would have  two poincar\'{e} duals namely, $c_1^2$ and $c_2,$ a contradiction. Thus,  $c_1^i \neq c_1^{i-2}c_2$ for $3 \leq i \leq r.$
		
		Suppose that $c_1^nc_2\neq 0.$ As rk $ H^*(F)=2n+2,$  we must have $r=n+2,$ and $c_1^{n+1}=c_1^{n-1}c_2,$ a contradiction.

		Next, suppose that $c_1^nc_2= 0.$  As the cup product $c_1c_2\neq 0,$ we get $c_1^r=c_1^{r-2j}c_2^j,j>1$ which generates  $H^r(F).$  Thus,    $c_1^{2i}\neq c_2^i$ for $1\leq i \leq j-1$  and $c_1^{2j}= c_2^j$ where $2j<r.$ We get rk $H^*(F)=j(r-2j)+2j+j$ which must be  $2n+2.$ So, $r= \frac{2n+2}{j}+2j-3,$  where $n+1= jk,$ for some $k\in \mathbb{N},$ or $j=2.$  Clearly, $c_1^{r-2j+1}c_2=0.$ Hence, the cohomology ring $H^*(F)$ is generated by $c_1, c_2$ with $c_1^{r+1}=c_1^{2j}+c_2^{j}=c_1^{r-2j+1}c_2=0,$ with $\frac{3j}{2}-1<n.$ This realizes possibility (5) for $q=2.$

		\textbf{Subcase (ii):} Assume that $c_1^2=c_2.$ Recall that $H^*(F)$ has at most two generators. \\ If $H^*(F)$ has  one generator, then it is generated by $c_1.$ As rk $H^*(F)=2n+2,$ we get    $c_1^{2n+1}\neq 0$ and $c_1^{2n+2}= 0.$ Hence, $F\sim_{2} \mathbb{R}P^{2n+1}.$ This realizes possibility (2) for $q=1.$ \\ If $H^*(F)$ has two generators, then $j^*(\alpha^n)= ( t \otimes c_1+1 \otimes c_1^2 )^n= \sum_{r=0}^{n} {{n}\choose{r}}t^{r}\otimes  c_1^{2n-r}.$ So, $c_1^n$ must be non zero. Now, we divide this subcase according as  $c_1^{n+1}$ is zero or nonzero.\\
		First, we assume that $c_1^{n+1}=0.$ 
		
		In this case, we have  $j^*(\alpha^n)=t^n \otimes c_1^n .$ 
		Let $d=d_{m-k}=d_q$ be the least degree element such that $d_{q}\neq {c_1}^i$ for some $i,1 \leq i \leq n.$ We have $d_i=0~\forall ~i>m.$  By the similar calculations done in Case (1),  we evaluate $d_i's$ for $i\leq m.$   For $q\leq n,$  $d_i=c_1^i$~ if $1 \leq i \leq q-1,$ $d_i=c_1^{i-q} d+c^i$ if $q+1\leq i \leq n~\&~$$d_i=c_1^{i-q}d$ if $n+1 \leq i \leq n+q,$ and zero otherwise. And for $q>n,$  $d_i=c_1^i,~\mbox{if}~1 \leq i \leq n,d_i=0,~\mbox{if}~ n+1\leq i \leq q-1~\&~ d_i=c_1^{i-q}d, ~\mbox{if}~q+1\leq i \leq q+n,$ and zero otherwise. We have $j^*(\alpha^n \beta)=\sum_{i=0}^{m-1}t^{i+n}\otimes d_{m-i}c_1^n.$ As $c_1^{n+1}=0,$ we get $j^*(\alpha^n\beta)=t^{k+n}\otimes c_1^nd.$ Clearly, $ c_1^{n}d\neq 0.$ If $d^2=0,$ then $F\sim_2 \mathbb{R}P^n \times \mathbb{S}^q, 1\leq q \leq m.$ This realizes possibility (1) for $\mathbb{F}=\mathbb{R}.$ If $d^2 \neq 0,$ then we have either $d^2=c_1^qd$ or $d^2=c_1^{2q}.$\\ 
		\indent Let $d^2=c_1^qd.$ As done earlier, $q$ must be $1.$ By the change of basis $d'=d+c_1,$ we get $d'^{n+1}+d^{n+1}=d'^{n+2}=d^{n+2}=dd'=0$. Thus, $F\sim_2 \mathbb{R}P^{n+1}\# \mathbb{R}P^{n+1}.$ This realizes possibility (2) for $\mathbb{F}=\mathbb{R}.$ 
		
		Let $d^2=c_1^{2q}.$ This is possible only when $2q\leq n.$ By the change of basis $d'=d+c_1^q,$ we get $d'^2=0$ and $c_1^nd'=c_1^nd\neq 0.$ Thus, $F\sim_2 \mathbb{R}P^n\times \mathbb{S}^q, 1 \leq q \leq m.$ This realizes possibility (1) for $\mathbb{F}=\mathbb{R}.$\\
		Now, we assume that $c_1^{n+1}\neq 0.$

		In this case, $j^*(\alpha^n\beta)=\sum_{r=0}^{n}\sum_{i=0}^{m-1}{{n}\choose{r}}t^{r+i}\otimes c_1^{2n-r}d_{m-i}.$ Let   $d_{m-k}=$ $d_q=d$ be the least degree element  which is not a power of $c_1.$ Clearly, $d_{m-i}=\bigoplus_{l+qj=m-i} A_{l,j}c_1^ld^j,$ where $l\geq 0,j \in \{0,1\}, l=j\neq 0.$  So, we get \begin{equation*}
			j^*(\alpha^n \beta )=\sum_{r=0}^{n}\sum_{i=1}^{m-1}{{n}\choose{r}}t^{r+i}\otimes c_1^{2n-r+m-i}+\sum_{r=0}^{n}\sum_{i=q}^{m-1}A_{l,j}{{n}\choose{r}}t^{r+i}\otimes (\oplus_{l+q=m-i}c_1^{2n-r+l}d).
		\end{equation*}
		Note that the least degree terms in the image of $j^*({\alpha^n\beta})$ are $c_1^{n+1}$ and $c_1^nd.$
		
		Further, if the cup product $c_1d=0,$ then we must have $d^2\neq 0, $ otherwise, we cannot have poincar\'{e} dual of $d,$ a contradiction. It is easy to observe that   $c_1^r=d^{\frac{r}{q}}$ is the generator of $H^r(F),$ where $r$ is the formal dimension of $H^*(F).$ As rk$H^*(F)=2n+2,$ we get $r= \frac{q(2n+2)}{q+1}.$ So, $(q+1)|(2n+2),$ and hence, $n=(q+1)k-1$ for $q$ even $\&$ $n=(\frac{q+1}{2})k-1$ for $q$ odd.
		Thus, $c_1^{r+1}=c_1^r+d^{\frac{r}{q}}=c_1d=0.$ This realizes possibility (4).	
		
		Next, if the cup product $c_1d$ is nonzero, then either $c_1^nd \neq 0$ or $c_1^nd=0.$ Let $r$ be the formal dimension of $F.$ In this case,  
		we show that the generators of $H^r(F)$ would be $c_1^{r}$ which is equal to   $c_1^{r-qj}d^j,$ where $j>1.$  Assume that $c_1^{r}= 0.$  Now, if $c_1^{r-q}d$ is generator of $H^r(F),$ then $c_1^i \neq c_1^{i-q}d, q <i <r,$ otherwise, we get $c_1^r=c_1^{r-q}d$, a contradiction. Thus, rk $H^*(F)\geq 2n+4>2n+2,$ which contradicts our hypothesis. Similarly, $c_1^{r-qj}d^j,$  $j>1,$ cannot be generator of $H^r(F).$ Thus, $c_1^r\neq 0.$  Further, if $c_1^r=c_1^{r-q}d$ is a generator of $H^r(F),$ then $c_1^{r-q}$  would have two poincar\'{e} duals namely, $c_1^q$ and $d,$ again a contradiction. Hence, $c_1^r=c_1^{r-qj}d^j$ is generator of $H^r(F),$ where $j>1.$
		
		Let $c_1^nd \neq 0.$ It is easy to observe that rk $H^*(F) >2n+2,$ a contradiction.
		
		Let $c_1^nd=0.$ In this case, we must have  $q\leq n.$   As cup product $c_1d\neq 0,$ we get $c_1^r=c_1^{r-qj}d^j,j>1$ which generates  $H^r(F).$  Thus,    $c_1^{qi}\neq d^i$ for $1\leq i \leq j-1$  and $c_1^{qj}= d^j,$ where $qj<r.$ We get rk $H^*(F)=j(r-qj)+qj+j$ which must be  $2n+2$ so, $r= \frac{2n+2}{j}+qj-(q+1),$ where $n+1= jk,$ for some $k\in \mathbb{N},$ or $j=2.$  Hence, the cohomology ring $H^*(F)$ is generated by $c_1, d$ with $c_1^{r+1}=c_1^{qj}+d^{j}=c_1^{r-qj+1}d=0.$ As $qj<r,$ we get  $\frac{(q+1)j}{2}-1<n.$ This realizes possibility (5).

		\noindent Finally, if $c_2 =0$ or  \{$c_2 \neq 0~\&~ B_2=0$\},  then by Proposition \ref{prop 2.4}, this case holds only for  $n=1.$ So,   $j^*(\alpha )= t \otimes c_1.$   Thus, we have   rk $H^*(F)=4.$  Clearly,  for one generator $F\sim_2 \mathbb{R}P^3.$ Suppose  $H^*(F)$ has two generators. Let $d_q$ be the least degree element such that $d_q\neq c_1.$ Then, for $q>1,$ $j^*(\alpha \beta )=t^{m-q}\otimes c_1d_q\neq 0.$ Thus, we must have  $F\sim_2 \mathbb{S}^1\times \mathbb{S}^q, 1< q \leq m .$  For $q=1,$ if  $d_1^2=0=c_1^2,$  then we get $F\sim_2 \mathbb{S}^1\times \mathbb{S}^1$ and if $d_1^2=c_1d_1$ or $c
		_1^2=c_1d_1,$ then by change of basis $d'=c_1+d_1,$ we get  $F\sim_2 \mathbb{R}P^2 \# \mathbb{R}P^2.$ This realizes possibilities (1),(3) and (4) for $n=1.$
	\end{proof} 
	\begin{remark}
		For $n=1,$ we get $X \sim_2 \mathbb{S}^2 \times \mathbb{S}^m.$ By Theorem \ref{thm 3.2},  the possibilities of connected  fixed point sets $F$ of involutions on $X$ are   $\mathbb{S}^r \times \mathbb{S}^q, 1\leq r \leq 2~\&~1 \leq q \leq m,$ or $\mathbb{R}P^3,$ or  $ \mathbb{F}P^2 \# \mathbb{F}P^2, \mathbb{F}=\mathbb{R}$ or $\mathbb{C}.$ These possibilities have also been realized in [Theorem 3.11, \cite{j1}].
	\end{remark}	 
	

	\noindent Now, we determine the fixed point sets of involutions on $X$  when  $X$ is not TNHZ in $X_G$ under some assumptions on the associated Lerey-Serre spectral sequence of Borel fibration $ X \hookrightarrow X_G \rightarrow B_G.$   
	\begin{theorem}\label{3.10}
		Let $G=\mathbb{Z}_2$ act on a finite CW-complex $X \sim_2 \mathbb{F}P^n \times \mathbb{S}^m,$ where $\mathbb{F}=\mathbb{R}$ or $
		\mathbb{C},$ and $X$ is not TNHZ in $X_G.$ If the associated Leray-Serre spectral sequence is nondegenerate and the differentials $d_r$ of the spectral sequence satisfies $d_r(1\otimes b)=0~ \forall ~r \leq m.$   Then, $F\sim_2 \mathbb{S}^q,$ where $-1\leq q \leq ln+m$ and $l=1$ or $2$ for $\mathbb{F}=\mathbb{R}$ or $\mathbb{C},$ respectively.
		
	\end{theorem}
	\begin{proof}
		As $E_2 \neq E_{\infty}$ and 
		$d_r(1\otimes b) =0~~\forall ~r\leq m,$ we must have either $d_{l+1}(1\otimes a)\neq 0$  or $d_{m+1}(1\otimes b) \neq 0,$ where deg $a=l$ and deg $b=m.$ We consider two cases according as $\pi_1(B_G)$ acts trivially or nontrivially on $H^*(X).$\\
		First, suppose that $\pi_1(B_G)$ acts trivially on $H^*(X).$  Clearly, if $\{d_{l+1}(1\otimes a) =0~\&~d_{m+1}(1\otimes b) \neq 0\}$ or $\{d_{l+1}(1\otimes a) \neq 0~\&$ $n$ is odd\}, then no line of the spectral sequence survive to $E_{\infty}$-term. Thus, $F = \emptyset.$ This realizes theorem for $q=-1.$  If  $d_{l+1}(1\otimes a)\neq 0$ and $n$ is even, then  two lines survive to $E_{\infty}$-term. Thus, we get rk $H^k(X_G)=$ rk $H^k(F_G)= $rk $H^*(F)=2,$  for sufficiently large $k.$ Hence, $F \sim_2 \mathbb{S}^q, 0\leq q \leq ln+m,$  where deg $a =l,$ $l=1$  or $2$ for $\mathbb{F}=\mathbb{R}$ or $\mathbb{C},$ respectively.\\
		Next, suppose that $\pi_1(B_G)$ acts nontrivially on $H^*(X).$ This is possible only when $m \leq ln$ and  $m\equiv 0$(mod $l$). Now, we consider two cases (i) $l=m< ln$ (ii) $l<m\leq ln.$ \\
		\textbf{Case(i)}: Assume that $l=m< ln.$ If $n$  is even, then the orders of $a,b$ and $a+b$ are $n+1,2$ and $n+2,$ respectively. Thus, $g^*=1 $ on $H^*(X),$ which contradicts our hypothesis. If $n$ is odd, then the orders of $a,b$ and $a+b$ are $n+1,2$ and $n+1,$ respectively. Thus, the only nontrivial action of $g^*$ on $H^m(X)$ is define by $g^*(b)=b$ and $g^*(a)=a+b.$ By Proposition \ref{E_2}, we get 
		\[
		E_2^{0,li}\cong 
		\begin{cases}
			\mathbb{Z}_2=<a^{i-1}b> &  k=0, 1\leq i(odd)\leq n \\
			\mathbb{Z}_2\oplus \mathbb{Z}_2 =<a^i,a^{i-1}b> & k\geq 0,~1<i(even)< n\\
			\mathbb{Z}_2 & k\geq 0~\&~i=0,n+1.
		\end{cases}
		\]
		By Proposition \ref{prop 5}, we have $ag^*(a)=a^2+ab$ is a permanant cocycle. If $d_{l+1}(b)= t^{l+1}\otimes 1,$ then $0=d_{l+1}(b(t\otimes 1))=t^{l+1}\otimes 1,$ a contracdiction. So, we get $d_{l+1}(b)=0.$ As $E_2 \neq E_{\infty}$ and  $a^2+ab$ is permanent cocycle,  we must have  $d_{2l+1}(a^2)=d_{2l+1}(ab)=t^{2l+1}\otimes 1.$ Consequently, for $2 \leq i(even)\leq n-1,$	
		\[
		d_{2l+1}(a^{i})= 
		\begin{cases}
			0 &  i \equiv 0 ~(mod ~4) \\
			(t^{2l+1}\otimes 1)a^{i-2} & i \not\equiv 0 ~(mod ~4)
		\end{cases}
		\]
		and for $1\leq i(odd) \leq n,$
		\[
		d_{2l+1}(a^{i}b)= 
		\begin{cases}
			(t^{2l+1}\otimes 1)(a^{i-1}+a^{i-2}b) &  i+1 \equiv 0 ~(mod ~4) \\
			(t^{2l+1}\otimes 1)a^{i-1} & i+1 \not\equiv 0~ (mod ~4)
		\end{cases}
		\]
		So, we have  $E_{2l+2}^{*,*}=E_{\infty}^{*,*}.$ Thus, no line  of the spectral sequence survive to $E_{\infty}$-term. Hence, $F = \emptyset.$  \\
		\textbf{Case(ii):}  Assume that $l<m\leq ln, m \equiv 0$(mod $l$). Further, if $l<m<2m\leq ln,$ then order of $a^{\frac{m}{l}}+b$ is not equal to $2.$ This implies that $\pi_1(B_G)$ acts trivially on $H^*(X).$ So, we consider $l<m\leq ln<2m,$ then we must have $g^*(a^\frac{m}{l})=a^{\frac{m}{l}}$ and $ g^*(b)=a^{\frac{m}{l}}+b.$
		By Proposition \ref{E_2}, the $E_2$-term of Leray-Serre spectral sequence is given by
		\[
		E_2^{k,i}=
		\begin{cases}
			\mathbb{Z}_2 & \mbox{for}~ ~k\geq 0, 0\leq i\equiv 0~(\mbox{mod } l) \leq m-l ~\& ~ ln<i \equiv 0~(\mbox{mod } l)\leq ln+m \\
			\mathbb{Z}_2 & \text{for } k=0, ~ m\leq i \equiv 0~(\mbox{mod } l) \leq ln\\
			0& \mbox{otherwise},
			
		\end{cases}
		\] where $ l=1$ and $2$ for $\mathbb{F}=\mathbb{R}$ and $\mathbb{C},$ respectively.
		As $E_2\neq E_\infty,$ we must have  $d_{l+1}(1\otimes a) = t^{l+1} \otimes 1.$ Clearly, if $m\equiv 0$ (mod $2l$), then  $F= \emptyset.$ If  $m \not \equiv 0$ (mod $2l$), then two lines of the spectral sequence survives to infinity, and hence $F \sim_2 \mathbb{S}^q,$ for $ 0\leq q \leq ln+m.$ Hence, our claim.	
	\end{proof}
	
	Now, we give examples to realizes above Theorems.
	\begin{example}
		Let $G=\mathbb{Z}_2$ act  on $\mathbb{S}^m$ define by $$(x_0,x_1,\cdots x_m) \mapsto (x_0,x_1,\cdots,x_q, -x_{q+1}, \cdots  -x_m) .$$
		If we consider trivial action of  $G$ on $ \mathbb{F}P^n,$   then  after taking diagonal action of $G$ on $\mathbb{F}P^n \times \mathbb{S}^m,$ we get fixed point set is $\mathbb{F}P^n \times \mathbb{S}^q,$ where $ 1\leq q \leq m$ and $ \mathbb{F}=\mathbb{R}$ or $\mathbb{C}.$\\
		If we take conjugation action of  $G$ on $ \mathbb{C}P^n$  i.e $(z_0,z_1,\cdots z_n) \mapsto (\bar{z}_0,\bar{z}_1,\cdots\bar{ z}_n),$ then after taking diagonal action of $G$ on $\mathbb{C}P^n \times \mathbb{S}^m,$ we get fixed point set is $\mathbb{R}P^n \times \mathbb{S}^q,1\leq q \leq m.$ 	\\	 
		This examples realizes possibility (1) of Theorems \ref{thm 3.1} and  \ref{thm 3.2}.
	\end{example}
	\begin{example}
		Bredon (\cite{G.Bredon}) constructed an example for fixed point set $\mathbb{P}^{2}(q)\# \mathbb{P}^2(q),$ (connected sum of projective spaces) of an involution on $\mathbb{S}^q \times 
		\mathbb{S}^{q+k},$ where $q=1,2$ and $k\geq q.$ This example also  realizes possibility (2) of Theorems \ref{thm 3.1} and  \ref{thm 3.2}, for $n=1.$ In this paper, he also gave  examples of involutions on $X\sim_2 \mathbb{S}^n \times \mathbb{S}^m,n\leq m$ with $F=\mathbb{R}P^3,$ and on  $X\sim_2 \mathbb{S}^n \times \mathbb{S}^m,n=1,2~\&~m\geq 2n-1$ with $F\sim_2 \mathbb{S}^{2n-1}.$  These examples also  realizes possibility (3) of Theorems \ref{thm 3.2} and Theorem \ref{3.10},  for $n=1,$ respectively.  
		
	\end{example}
	
	\section{Cohomology  Ring of  The Orbit Space of Free involutions on Product of Projective Space and Sphere}

	In this section, we compute the cohomology ring of the orbit space of free involutions on a space $X$ having mod $2$ cohomology the  product of  projective space and sphere $ \mathbb{F}P^n \times \mathbb{S}^m,$  $\mathbb{F}=\mathbb{R}$ or $\mathbb{C}.$ For the existence of free involutions on $ \mathbb{F}P^n \times \mathbb{S}^m,$ consider the diagonal action on $\mathbb{F}P^n\times \mathbb{S}^m,$ by taking any involution on 	$\mathbb{F}P^n$ and antipodal action on $\mathbb{S}^m.$
	First, we prove the following lemma.
	\begin{lemma}\label{4.1}
		Let $G=\mathbb{Z}_2$ act freely on $X\sim_2 \mathbb{F}P^n\times \mathbb{S}^m,$ where $\mathbb{F}=\mathbb{R}$ or $\mathbb{C},$ and $n,m\geq 1.$ Then, $\pi_1(B_G)$  acts trivially on $H^*(X)$ whenever one of the following holds: (1) $ ln \leq m$ (2) $ l=m<ln, n$ is even, (3) $l<m<2m\leq ln, m\equiv 0$(mod $l$), where $l=1$ or $2$ for $\mathbb{F}=\mathbb{R}$ or $\mathbb{C},$ respectively, and (4) $m\not \equiv 0$(mod $2$) for $\mathbb{F}=\mathbb{C}.$
	\end{lemma}
	\begin{proof}
		Let $g$ be the generator of $\pi_1(B_G).$  First, we consider the case $ln=m. $ Assume that $g^*\neq 1 $ on $H^m(X).$ We must have  $g^*(a^m)= a^m $ and  $g^*(b)=a^m+b.$ So, we get $bg^*(b)=a^mb \neq 0.$ By Proposition \ref{prop 5}, the fixed point set is non empty which contradicts our hypothesis. Clearly, if $ln<m$ or   \{$l=m<ln$ and $ n $ is even\}, then $g^*=1 $ on $H^*(X).$  Now, consider  $l<m<2m\leq ln,$ where $m \equiv 0$(mod $l$). Then, $g^*(a^{\frac{m}{l}})=a^{\frac{m}{l}}.$ If  $g^*(b)=a^{\frac{m}{l}}+b,$ then we must have  $a^{{\frac{2m}{l}}}=0,$ which is not  possible. So,  $g^*= 1 $ on $H^*(X).$ Finally, for $\mathbb{F}=\mathbb{C}$ if $m\not \equiv 0 $(mod $2$), then $H^i(X)\cong \mathbb{Z}_2~\forall ~0\leq i\equiv 0($mod $ 2) \leq 2n ~\&~ i=m+q, 0\leq q\equiv 0($mod $2) \leq 2n.$ Thus, $g^*=1$ on $H^*(X).$ 
	\end{proof}
	Now, we determine the cohomology ring of  orbit space $X/G,$ when $\pi_1(B_G)$ acts nontrivally on $H^*(X).$
	\begin{theorem}\label{thm 4.2}
		Let $G=\mathbb{Z}_2$ act freely on a finite CW-complex $X\sim_2\mathbb{F}P^n\times \mathbb{S}^m,$ where   $\mathbb{F}=\mathbb{R}$ or $\mathbb{C}$ and $ n, m\geq 1$ Assume that $\pi_1(B_G)$ acts nontrivially on $H^*(X).$ Then, $H^*(X/G)$ is isomorphic to one of the following graded commutative  algebras:
		\begin{enumerate}
			\item ${\mathbb{Z}_2[x,y,z]}/{<x^{2l+1},y^2+a_0z+a_1x^{2l},z^{\frac{n+1}{2}}+a_2x^{2l}z^{\frac{n-1}{2}},xy>,}$
			where  deg $x=1,$ deg $y=l~\&$ deg $z=2l, a_i\in \mathbb{Z}_2, 0\leq i \leq 2,m=l<ln,$  $n$ odd, and 
			\item ${\mathbb{Z}_2[x,y,z,w_{k}]}/{<x^{l+1},y^{\frac{m}{2l}}+a_0w_1,z^{2},xw_{k},{w_kw_{k+i}}+a_{k,i}x^{dl}y^{\frac{2m-ln+lq}{2l}}z>}, ~$
			where$~$ deg $x=1,$ deg $y=2l,$ deg $z=ln+l,$ deg $w_{k}=m+l(k-1),$ $1\leq k\leq \frac{ln-m+l}{l};$  $0\leq i\leq \frac{ln-m}{l};$  $-1\leq q(odd)\leq \frac{ln-m-2l}{l} ;d=0,1;$ $ $$l<m<ln<2m,m\equiv 0$ (mod $2l$), $n$  odd; and $a_{k,i}=0$ if $\frac{l(n+2)-m}{l}<2k+i,a_0~\& $ ${a_{k,i}}'s$ are in $\mathbb{Z}_2.$
			If $d=0,$ then $i$ is even and $q=2k+i-3.$ If $d=1,$ then $i$ is odd and $q=2k+i-4,$ 
		\end{enumerate}
		where  $l=1$ or $2$ for $\mathbb{F}=\mathbb{R}$ or $\mathbb{C},$ respectively,
	\end{theorem}
	\begin{proof}
		By Lemma \ref{4.1}, it is clear that $\pi_1(B_G)$ can act nontrivially on $H^*(X)$ only when (1) $l=m<ln,n$ is odd, and (2) $l<m<ln<2m, m\equiv 0$ (mod $l$), where $l=1$ and $2$ for $\mathbb{F}=\mathbb{R}$ and $\mathbb{C},$ respectively. \\
		First, we consider $l=m<ln,$ where $n$ is odd. By Theorem \ref{3.10}, we get $E_{2l+2}^{*,*}=E_{\infty}^{*,*}.$ Clearly, $E^{k,li}_{\infty}\cong \mathbb{Z}_2, 0 \leq k \leq 2l, 1<i(even)<n ~\&~ k=0, 1\leq i(odd)\leq n ;$ and trivial otherwise. So,  we have  	
		
		\[
		H^j(X_G)=
		\begin{cases}
			\mathbb{Z}_2 \oplus \mathbb{Z}_2  &  0< j\equiv 0(\mbox{mod} ~l) \leq ln\\
			\mathbb{Z}_2 & 0<j\not \equiv 0(\mbox{mod}~ l) <ln+l~\&~j=0,ln+l\\
			0 & \text{otherwise.}
		\end{cases}
		\]
		Let $x\in E^{1,0}_{\infty},$ $u \in E^{0,l}_{\infty}$ and $v \in E_{\infty}^{0,2l}$    be the elements corresponding to permanent cocycles $t\otimes 1\in E_2^{1,0} ,$ $b \in E_2^{0,l}$ and  $a^2+ab \in E_2^{0,2l},$  respectively. Clearly, $x^{2l+1}=u^2=v^{\frac{n+1}{2}}=xu=0.$ Thus, the total complex $$\mbox{Tot} ~E_{\infty}^{*,*}={\mathbb{Z}_2[x,u,v]}/{<x^{2l+1},u^2,v^{\frac{n+1}{2}},xu>}, $$ where deg $x=1,$ deg $u=l$ and deg $v=2l.$ Suppose that  $y\in H^l(X_G)$  and  $z\in H^{2l}(X_G)$ corresponding to $u$ and $v$ such that $i^*(y)=b$ and  $i^*(z)=a^2+ab,$ respectively. Clearly, we have  $x^{2l+1}=y^2+{a_0}z+a_1x^2=z^{\frac{n+1}{2}}+a_2x^{2l}z^{\frac{n-1}{2}}=xy=0.$ Hence, the cohomology ring $ H^*(X_G)$ is given by 	$${\mathbb{Z}_2[x,y,z]}/{<x^{2l+1},y^2+a_0z+a_1x^2,z^{\frac{n+1}{2}}+a_2x^{2l}z^{\frac{n-1}{2}},xy>,}$$ 
		where deg $x=1,$ deg $y=l$ and deg $z=2l, a_0,a_1,a_2 \in \mathbb{Z}_2, m=l<ln, n$ odd. As $G$ act freely on $X,$ by Proposition \ref{mac}, $H^*(X_G)\cong H^*(X/G).$  This realizes possibility (1).\\
		Next, we consider  $l<m<ln<2m,$ where $m\equiv 0$ (mod $l$). Note that $E_2$-term of the Lerey-Serre spectral sequence is same as the Case (ii) of Theorem \ref{3.10}. By Proposition \ref{prop 5}, we have $bg^*(b)=a^{\frac{m}{l}}b$ is a permanant cocycle. As $G$ acts freely on $X,$  we must have $d_{l+1}(1\otimes a)=t^{l+1}\otimes 1.$ Consequently, $d_{l+1}(1\otimes a^{m-l})=t^{l+1}\otimes a^{m-2l}$ if $m-l \not \equiv 0 $ (mod $2l$). If $m-l  \equiv 0 $ (mod $2l$), then two lines survives to infinity which is  not possible. Thus, 	$d_{l+1}(1\otimes a^{n+2-{\frac{m}{l}}}b)=t^{l+1}\otimes a^{n+1-{\frac{m}{l}}}b.$ Consequently, $d_{l+1}(1\otimes a^nb)=t^{l+1}\otimes a^{n-1}b.$ As $2m>ln$ and $ a^{\frac{m}{l}}b$ is permanent cocycle,  we get $a^{\frac{m}{l}}b=a^{n+j-{\frac{m}{l}}}b,$ where $0\leq j \leq \frac{m}{l}.$ This implies that $n$ must be odd. Thus, $E_{l+2}^{*,*}=E_{\infty}^{*,*},$ and hence  $E^{k,li}_{\infty}\cong \mathbb{Z}_2, 0 \leq k \leq l, 1\leq i(even)\leq \frac{m}{l}-2 ~\&~ i=n+j,  1\leq j(odd)\leq \frac{m}{l} -1;$ $E^{0,li}_{\infty}\cong \mathbb{Z}_2, $$\frac{m}{l}\leq i\leq n,$ and trivial otherwise. Thus, for $\mathbb{F}=\mathbb{R},$ the cohomology groups are given by   	\[
		H^j(X_G)=
		\begin{cases}
			\mathbb{Z}_2   &  0\leq  j \leq  n+m,\\
			0 & \text{otherwise.}
		\end{cases}
		\]
		For $\mathbb{F}=\mathbb{C},$ we have
		\[
		H^j(X_G)=
		\begin{cases}
			\mathbb{Z}_2   &  0\leq j \leq m-2, j\not = 2d+1, d \mbox{ ~odd}; m\leq j(\mbox{even})\leq 2n,\\
			\mathbb{Z}_2 &  2n+2 \leq j \leq 2n+m, 	 j	\not = 2n+2d+3, d \mbox{  ~odd},\\
			0 & \text{otherwise.}
		\end{cases}
		\]
		Let $x \in E_{\infty}^{1,0},$ $u \in E^{0,2l}_{\infty},$ $v_{k} \in E_{\infty}^{0,m+lk-l}, 1\leq k \leq \frac{ln-m+l}{l}$    and $s \in E_{\infty}^{0,ln+l}$ be the elements corresponding to the permanent cocycle  $t\otimes 1 \in  E_2 ^{1,0},$ $1\otimes a^2 \in E_2^{0,2l},$ $a^{\frac{m+lk-l}{l}} \in E_2^{0,m+lk-l},$ and   $1\otimes a^{n+1-\frac{m}{l}} \in E_2^{0,ln+l},$  respectively. Assume that $y\in H^{2l}(X_G)$ determines  $u$ such that $i^*(y)=a^2,$  $z\in H^{2l+l}(X_G)$ determines $v$  such that $i^*(z)=a^{n+1-\frac{m}{l}}$ and $w_{k}\in H^{m+lk-l}(X_G)$ determines $v_{k}$ such that $i^*(w_{k})=a^{\frac{m+lk-l}{l}}.$ Hence, we have $x^{l+1}=y^{\frac{m}{2l}}+a_0w_1=z^{2}=xw_{k}={w_kw_{k+i}}+a_{k,i}x^{dl}y^{\frac{2m-ln+lq}{2l}}z=0,$ where  $0\leq i\leq \frac{ln-m}{l},$ $-1\leq q(odd)\leq \frac{ln-m-l}{l} ~\&~ d=0,1.$ Also, if $\frac{l(n+2)-m}{l}<2k+i,$ then $a_{k,i}=0.$
		If $d=0,$ then $i$ is even and $q=2k+i-3.$ If $d=1,$ then $i$ is odd and $q=2k+i-4.$  Thus, the cohomology ring  $H^*(X_G)$ is given by  $${\mathbb{Z}_2[x,y,z,w_{k}]}/{<x^{l+1},y^{\frac{m}{2l}}+a_0w_1,z^{2},xw_{k},{w_kw_{k+i}}+a_{k,i}x^{dl}y^{\frac{2m-ln+lq}{2l}}z>},$$ where  deg $x=1,$ deg $y=2l,$ deg $z=ln+l,$ deg $w_{k}=m+l(k-1),$ $1\leq k\leq \frac{ln-m+l}{l}, a_0 ~\&~ {a_{k,i}}'s$ are in $\mathbb{Z}_2,$ $n$  odd and $m\equiv 0$ (mod $2l$).
		This realizes possibility (2).
	\end{proof}
	\begin{remark}
		If  $a_i=0, 0\leq i \leq 2$ in possibility (1) of the Theorem \ref{thm 4.2}, then $X/G\sim_2  (\mathbb{R}P^{2l} \vee \mathbb{S}^l) \times\mathbb{P}^{\frac{n-1}{2}}(2l).$  
	\end{remark}
	\begin{remark}
		In particular, for $\mathbb{F}=\mathbb{R}$ $\&~m=2,$ this result has been discussed in case (1) of Theorem 3.5 \cite{nontrivial}. Theorem \ref{thm 4.2} generalises this result for all $m \geq 1.$
	\end{remark}
	Finally, we discuss the cohomology ring of orbit spaces of free involutions on $X,$ when $\pi_1(B_G)$ acts trivially on $H^*(X),$ under some assumptions on the associated Lerey-Serre spectral sequence of Borel fibration $ X \hookrightarrow X_G \rightarrow B_G.$   
	\begin{theorem}\label{thm 4.3}
		Let $G=\mathbb{Z}_2$ act freely on finite CW-complex $X\sim_2\mathbb{F}P^n\times \mathbb{S}^m,$ where $\mathbb{F}=\mathbb{R}$ or $\mathbb{C}$ and $n,m\geq 1.$ Assume that $\pi_1(B_G)$ acts trivially on $H^*(X)$ and the differentials $d_r(1\otimes b)=0~\forall~ r\leq m.$ Then, the  cohomology ring of orbit space $H^*(X/G)$ is isomorphic to one of the following  graded commutative algebras: 	\begin{enumerate}
			\item $\mathbb{Z}_2[x,y,z]/I,$ where $I$ is homogeneous ideal given by:
			\begin{center}
				$<x^{l+1},y^{\frac{n+1}{2}}+a_0y^{\frac{l(n+1)-m}{2l}}z+a_1x^{l}y^{\frac{ln-m}{2l}}z+a_2z,z^{2}+a_3x^{2i}y^{\frac{m-i}{l}}+a_{4}x^{i'}y^{\frac{m-{i'}}{2l}}z>,$ 
			\end{center}
			where  deg $x=1,$ deg $y=2l~\&$ deg $z=m;$ $a_0=0$ if $m \not\equiv 0 $(mod $2l$) or $m>ln+l;$ $a_1=0$ if $m \equiv 0 $(mod $2l$) or $m>ln;$ $a_2=0$ if $m \neq l(n+1);$
			$a_3=0$ if $m \not \equiv i$(mod $l$) or  \{$i=0$ and $2m > l(n-1)\}, 0\leq 2i\leq l$ and
			$ a_{4}=0$ if $m \not\equiv {i'} $(mod $2l$) or $m>ln, 0\leq {i'}\leq l,$  
			$a_k\in \mathbb{Z}_2, 0\leq k\leq 4,$  $n$ odd,	\vspace{0.5em}
			\item $\mathbb{Z}_2[x,y,z]/<x^{l+1},y^{\frac{n}{2}+1},z^2+a_0y+a_1x^lz>,$ where  deg $x=1,$ deg $y=2l~\&$ deg $z=m,$ $a_0,a_1\in \mathbb{Z}_2,$ $n$  even, and

			\vspace{0.5em}
			\item  $\mathbb{Z}_2[x,y]/<x^{m+1},y^{n+1}+{\sum_{0<i\equiv 0 (mod ~l)}^{min \{l(n+1),m\}}}a_ix^{i}y^{\frac{l(n+1)-i}{l}}>,$ 
			where deg $x=1,$ deg $y=l$ and $a_i\in \mathbb{Z}_2,$   
		\end{enumerate}
		where	$l=1$ or  $2$ for $\mathbb{F}=\mathbb{R}$ or  $\mathbb{C},$ respectively.
		
	\end{theorem}
	\begin{proof}
		As $\pi_1(B_G)$ acts trivially on $H^*(X),$ the $E_2$- term of  Leray-Serre spectral sequence of the fibration $X \hookrightarrow X_G \rightarrow B_G$ is given by $$E_2^{k,i}=H^k(B_G) \otimes H^i(X).$$ Since $G$  acts freely on $X$ and   differentials $d_r(1\otimes b)=0~\forall~ r\leq m,$ at least one of differentials $d_{l+1}(1\otimes a)$ and $d_{m+1}(1\otimes  b)$ must be non zero.\\
		First, consider $d_{l+1}(1\otimes a) =t^{l+1}\otimes 1$ and  $d_{m+1}(1\otimes b) =0,$ we get $d_{l+1}(1\otimes a^{i})=0$ \& $d_{l+1}(1\otimes a^{i}b) =0,$  when $0 \leq i(even)\leq n$ and $d_{l+1}(1\otimes a^{i}) =t^{l+1}\otimes a^{i-1}~\&$ $d_{l+1}(1\otimes a^{i}b)=t^{l+1}\otimes a^{i-1}b,$ when $0 < i(odd)\leq n.$ Clearly, for $n$ even two lines of the spectral sequence survive to infinity which follows that $F \neq \emptyset,$ a contradiction. Thus, $n$ must be odd. Clearly, $E_{l+2}^{*,*}=E_{\infty}^{*,*}.$ So, we  get:\\
		If $m\geq ln$ or $m\not \equiv 0$ (mod $2l$), then $E_{\infty}^{k,i}\cong \mathbb{Z}_2$ for $0\leq k \leq l, 0 \leq i \equiv 0$ (mod $2l$) $ \leq l(n-1);$ $i=m+i'$ where $0\leq i'\equiv 0$ (mod $l$) $\leq l(n-1),$  and trivial otherwise.\\ If $m<ln$ and $m\equiv 0$ (mod ~$2l$), then $E_{\infty}^{k,i}\cong \mathbb{Z}_2$ for $0\leq k \leq l,1\leq i\equiv 0$ (mod $2l$) $<m;$ $i=l(n+i'), 1\leq i'(odd)\leq \frac{m}{l}-1;$    $E_{\infty}^{k,i}\cong \mathbb{Z}_2\oplus \mathbb{Z}_2 $ for $0 \leq k \leq l, m\leq i\equiv 0$ (mod $2l$) $ \leq l(n-1),$ and trivial otherwise. \\Consequently, if $m>ln,$ then for $\mathbb{F}=\mathbb{R},$ the cohomology groups of $X_G$ are given by
		\[
		H^j(X_G)=
		\begin{cases}
			\mathbb{Z}_2   &  0\leq  j \leq  n , m \leq j \leq n+m  \\
			0 & \text{otherwise,}
		\end{cases}
		\]
		and for $\mathbb{F}=\mathbb{C},$ we have
		\[
		H^j(X_G)=
		\begin{cases}
			\mathbb{Z}_2   &  0\leq j \leq 2n, j\not = 2d+1, d \mbox{ ~odd} \\
			\mathbb{Z}_2   &	j =m+i\leq 2n+m,  i\not= 2d+1,  d \mbox{ ~odd}\\
			0 & \text{otherwise.}
		\end{cases}
		\]
		
		Now, if $\{m<ln$ and $m\equiv 0 $(mod $2l$)\} or $m=ln$, then for $\mathbb{F}=\mathbb{R},$ the cohomology groups are given by   	\[
		H^j(X_G)=
		\begin{cases}
			\mathbb{Z}_2   &  0\leq  j <  m; n+1 \leq j \leq n+m\\
			\mathbb{Z}_2\oplus \mathbb{Z}_2   &  m\leq  j \leq  n\\
			
			0 & \text{otherwise,}
		\end{cases}
		\]
		and for $\mathbb{F}=\mathbb{C},$ we have
		\[
		H^j(X_G)=
		\begin{cases}
			\mathbb{Z}_2   &  0\leq j <m , j\not= 2d+1, d \mbox{ ~odd};  j=2n+i, i	\not= 2d+3, d \mbox{  ~odd}\\
			\mathbb{Z}_2 \oplus \mathbb{Z}_2 &  m \leq j \leq 2n, j \not= 2d+1,  d \mbox{ ~odd}	\\ 
			0 & \text{otherwise.}
		\end{cases}
		\]
		Next, if $m<ln$ and $m\not \equiv 0$(mod $2l$), then for $\mathbb{F}=\mathbb{R},$ cohomology groups are same as in the case when  $m<ln$ and $m\equiv 0 $(mod $2l$). Simialrly, for $\mathbb{F}=\mathbb{C}$  the cohomology groups are given by $H^j(X_G)\cong \bigoplus_{i'+j'=j} E_{\infty}^{i',j'},$ where $0\leq i'\leq 2,$  and $0\leq j'\equiv 0 $(mod $4$) $\leq 2(n-1); j'=m+i, 0\leq i\equiv 0$(mod $4$)$\leq 2(n-1).$\\
		Suppose that the elements corresponding to permanent cocycles $t \otimes 1 \in E_2^{1,0},$ $1 \otimes a^2\in E_2^{0,2l}$and $1 \otimes b\in E_2^{0,m}$ are $x\in   E_{\infty}^{1,0},$ $u \in E_{\infty}^{0,2l}$ and $v \in E_{\infty}^{0,m},$ respectively.  Thus, the total complex is given by $$\mbox{Tot} ~E_{\infty}^{*,*}\cong \mathbb{Z}_2[x,u,v]/<x^{l+1}, u^{\frac{n+1}{2}}, v^2>,$$
		where deg $x=1,$ deg $u=2l$ and deg $v=m.$\\ Let  $y\in H^{2l}(X_G)$  and   $z\in H^{m}(X_G)$ determines $u$ and $v$  such that $i^*(y) =a^2$ and $i^*(z)=b,$ respectively. Clearly, $x^{l+1}=0. $ We also have $y^{\frac{n+1}{2}}+a_0y^{\frac{l(n+1)-m}{2l}}z+a_1x^{l}y^{\frac{ln-m}{2l}}z+a_2z=0,$ where $a_0=0$ if $m \not\equiv 0 $(mod $2l$) or $m>ln+l;$ $a_1=0$ if $m \not \equiv 0 $(mod $2l$) or $m>ln,$  $a_2=0$ if $m\neq l(n+1),$ and $z^{2}+a_{3}x^{i'}y^{\frac{m-{i'}}{2l}}z+a_{4}x^{2i}y^{\frac{m-i}{l}}=0,$ where $0\leq i'\leq l~\&~ 0\leq 2i\leq l,$  $a_3=0$ if $m \not \equiv i$(mod $l$),  also for $i=0,$   $a_3=0$ if  $2m > l(n-1),$ and
		$ a_{4}=0$ if $m \not\equiv {i'} $(mod $2l$) or $m>ln,$  $a_k\in \mathbb{Z}_2, 0\leq k \leq 4.$  
		Hence, the cohomology ring $H^*(X_G)\cong {\mathbb{Z}_2[x,y,z]}/{<I>},$ where $I$ is homogeneous ideal given by  \begin{center}
			$<x^{l+1},y^{\frac{n+1}{2}}+a_0y^{\frac{l(n+1)-m}{2l}}z+a_1x^{l}y^{\frac{ln-m}{2l}}z+a_2z,z^{2}+a_3x^{2i}y^{\frac{m-i}{l}}+a_{4}x^{i'}y^{\frac{m-{i'}}{2l}}z>,$ 
		\end{center}
		
		where   deg $x=1,$ deg $y=2l$ and  deg $z=m,~n $ odd.  This realizes possibility (1).\\
		Now, consider $d_{l+1}(1\otimes a) =d_{l+1}(1\otimes b) =t^{l+1}\otimes 1.$ So, $m=l.$ We get $d_{l+1}(1\otimes a^i)=0$ ~\&~ $d_{l+1}(1\otimes a^ib)= t^{l+1}\otimes a^i,$ when $0\leq i(even)\leq n$ and $d_{l+1}(1\otimes a^i)=t^{l+1}\otimes a^{i-1}$ ~\&~ $d_{l+1}(1\otimes a^ib)= t^{l+1}\otimes (a^i+a^{i-1}b),$ when $0<i(odd)\leq n.$ Thus,  $E_{l+2}^{*,*}=E_{\infty}^{*,*}.$ So, we have  $E_{\infty}^{k,i}\cong \mathbb{Z}_2$ for $0\leq k \leq l, 0 \leq i \equiv 0$ (mod $l$) $ \leq ln.$ Consequently, the cohomology groups are  given by   	
		\[
		H^j(X_G)=
		\begin{cases}
			\mathbb{Z}_2   &  j=0, ln+l; 0 <j\not \equiv 0 ~(mod ~l) <ln+l\\
			\mathbb{Z}_2\oplus \mathbb{Z}_2   &  0<  j  \equiv 0 ~(mod ~l) <ln+l\\
			
			0 & \text{otherwise.}
		\end{cases}
		\]
		
		Let  $x\in   E_{\infty}^{1,0},$ $u \in E_{\infty}^{0,2l},$ and $v \in E_{\infty}^{0,l}$ be elements corresponding to permanent cocyles $t \otimes 1 \in E_2^{1,0},$ $1 \otimes a^2\in E_2^{0,2l}$ and $1 \otimes (a+b)\in E_2^{0,l},$ respectively. For $n$ is odd, the total complex  Tot$E_{\infty}^{*,*}$ is same as in case (1). For $n$  even, $$\mbox{Tot} ~E_{\infty}^{*,*}\cong \mathbb{Z}_2[x,u,v]/<x^{l+1}, u^{\frac{n}{2}+1}, v^2>,$$
		where deg $x=1,$ deg $u=2l$ and deg $v=m.$\\ Let $y\in H^{2l}(X_G)$ and   $z\in H^{l}(X_G)$ determines  $u$  and $v$ such that $i^*(y)=a^2$  and $i^*(z)=a+b,$ respectively. For $n$ odd, the  cohomology ring of $X_G$ is same as in case (1) for $m=l.$ For $n$ even, we have  $x^{l+1}=y^{\frac{n}{2}+1}=z^2+a_0y+a_2x^{l}z=0, a_0,a_1, \in \mathbb{Z}_2.$  Thus, for $n$ even,  the cohomology ring is given by 
		$${\mathbb{Z}_2[x,y,z]}/{<x^{l+1},y^{\frac{n}{2}+1},z^2+a_0y+a_2x^{l}z>},$$ where  deg $x=1,$ deg $y=2l,$ deg $z=m$ and $a_0,a_1, \in \mathbb{Z}_2.$  This realizes possibility (2).\\
		Finally, consider $d_{l+1}(1\otimes a) =0 $ and  $d_{m+1}(1\otimes b) =t^{m+1}\otimes 1.$ We get, $d_{m+1}(1\otimes a^ib)=t^{m+1}\otimes a^i, 0<i\leq n.$ So, $E_{m+2}^{*,*}= E_{\infty}^{*,*},$ and hence $E_{\infty}^{k,i}\cong \mathbb{Z}_2, 0 \leq k \leq m ~\&~ 0 \leq i \equiv 0 $(mod $l$) $\leq ln.$ Thus, the cohomology groups of $X_G$ for $\mathbb{F}=\mathbb{R}$ are given by 
		\[H^j(X_G)=
		\begin{cases}
			\oplus	\mathbb{Z}_2(i+1)\mbox{copies}   &  j=i~\& ~j=n+m-i, 0 \leq i < min\{n,m\} \\
			\oplus \mathbb{Z}_2(min\{n,m\}+1) \mbox{copies}   & min\{n,m\}\leq j \leq n+m- min\{n,m\}\\
			
			0 & \text{otherwise.}
		\end{cases}
		\]

		Similarly, for $\mathbb{F}=\mathbb{C}$ the cohomology groups are given by $H^j(X_G)\cong \bigoplus_{i'+j'=j} E_{\infty}^{i',j'},$ where $0\leq i' \leq m,$ $0\leq j'\equiv 0$(mod $2$) $\leq 2n.$  Let $x\in E^{1,0}_{\infty}$ and $u \in E^{0,l}_{\infty}$   be the elements corresponding to permanent cocycles $t\otimes 1\in E_2^{1,0} $ and  $1 \otimes a \in E_2^{0,l},$ respectively.  Thus the total complex is given by  Tot$E_{\infty}^{*,*}\cong {\mathbb{Z}_2[x,y]}/{<x^{m+1},u^{n+1}>},$ where deg $x=1$ and deg $u=l.$ Now choose $y\in H^{l}(X_G)$ corresponding to $u$ such that $i^*(y)=a.$ Thus, we have $x^{m+1}=y^{n+1}+\sum_{0<i\equiv 0 (mod ~l)}^{min \{l(n+1),m\}}a_ix^{i}y^{\frac{l(n+1)-i}{l}}=0,$ where $a_i's$ are in $ \mathbb{Z}_2.$ Hence, the cohomology ring $H^*(X_G)$ is given by 	$${\mathbb{Z}_2[x,y]}/{<x^{m+1},y^{n+1}+\sum_{0<i\equiv 0 (mod ~l)}^{min \{l(n+1),m\}}a_ix^{i}y^{\frac{l(n+1)-i}{l}}>}$$ 
		where  deg $x=1$ and deg $y=l.$ This realizes possibility (3).
	\end{proof}
	\begin{remark}
		If  we take $a_i=0~\forall~ 0\leq i \leq 4,$ in possibility (1) of Theorem \ref{thm 4.3}, then $X/G\sim_2 \mathbb{R}P^l\times \mathbb{P}^{\frac{j}{2}}(2l) \times \mathbb{S}^m, j=n-1$ for $n$ odd and $j=n$ for $n$ even. If $a_i=0~ \forall~ 0<i\equiv 0(mod~l)\leq min \{l(n+1),m\}$ in the possibility (2), then $X/G\sim_2 \mathbb{R}P^m \times \mathbb{F}P^n,$ where $\mathbb{F}=\mathbb{R}$ or $\mathbb{C}.$  
	\end{remark}
	\begin{example}
		Let $T: \mathbb{F}P^n\rightarrow \mathbb{F}P^n$ be a map define as $[z_0,z_1,\cdots, z_{n-1},z_n] \mapsto [-\bar{z}_1,\bar{z}_0,\cdots, -\bar{z}_n,\bar{z}_{n-1}],$ when $n$ is odd.  Then, this gives a free involution on $ \mathbb{F}P^n,$ for $\mathbb{F}=\mathbb{R}$ and $\mathbb{F}=\mathbb{C}.$ The orbit spaces of these actions are  $\mathbb{R}P^n/\mathbb{Z}_2 \sim_2 \mathbb{S}^1\times \mathbb{C}P^{\frac{n-1}{2}} $ and  $\mathbb{C}P^n/\mathbb{Z}_2 \sim_2 \mathbb{S}^1\times \mathbb{H}P^{\frac{n-1}{2}}$ (\cite{hsingh}). The diagonal action define as $([z],x)\mapsto (T([z]),x)$ gives free involution on $\mathbb{F}P^n \times  \mathbb{S}^m.$   The orbit spaces are  $(\mathbb{R}P^n \times  \mathbb{S}^m)/\mathbb{Z}_2 \sim_2 \mathbb{S}^1\times \mathbb{C}P^{\frac{n-1}{2}} \times  \mathbb{S}^m$ and $(\mathbb{C}P^n \times  \mathbb{S}^m)/\mathbb{Z}_2 \sim_2 \mathbb{S}^1\times \mathbb{H}P^{\frac{n-1}{2}} \times \mathbb{S}^m,$  for $\mathbb{F}=\mathbb{R}$ and $\mathbb{F}=\mathbb{C},$ respectively. This realizes possibility (1) of Theorem \ref{thm 4.3}, for $a_k=0~\forall~ 0\leq k\leq 4.$ 
		
	\end{example}
	\begin{example}
		A map $\phi: \mathbb{F}P^n \times  \mathbb{S}^m \rightarrow \mathbb{F}P^n \times  \mathbb{S}^m, $ define as $([z],x)\mapsto ([z],-x)$ forms a free involution. Thus, the orbit space of this diagonal action is $(\mathbb{F}P^n \times  \mathbb{S}^m)/\mathbb{Z}_2 \sim_2 \mathbb{F}P^n \times  \mathbb{R}P^m.$ This realizes possibility (3), when $a_i=0~\forall ~i.$ The Dold manifold $D(n,m)$ is the orbit space of free involution on $\mathbb{C}P^n \times  \mathbb{S}^m $ define by  $([z],x)\mapsto ([\bar{z}],-x)$ \cite{peltier}. Recall that $H^*(D(n,m)) \cong  \mathbb{Z}_2[x,y]/<x^{m+1},y^{n+1}>,$ where deg $x=1,$ deg $y=2.$ This realizes possibility (3) of Theorem \ref{thm 4.3},  for $\mathbb{F}=\mathbb{C}$ and $a_i=0~\forall ~i.$
	\end{example}
	
	\section{Applications}
	In this section,  we derive the Borsuk-Ulam type results for free involutions on $X \sim_2 \mathbb{F}P^n \times \mathbb{S}^m ,$  where    $\mathbb{F}=\mathbb{R}$ or $\mathbb{C}.$ We determine the nonexistence of $\mathbb{Z}_2$-equivariant maps between  $ X$ and $ \mathbb{S}^{k},$ where $\mathbb{S}^k$ equipped with  antipodal actions.\\
	\indent Recall that \cite{floyd} the index (respectively, co-index) of a $G$-space $X$ is the greatest integer $k$ (respectively, the lowest integer $k$)  such that there exists a $G$-equivariant map $\mathbb{S}^k \rightarrow X$ (respectively, $ X \rightarrow  \mathbb{S}^k$).\\
	\indent	By  Theorems \ref{thm 4.2} and  \ref{thm 4.3}, we get the largest integer  $s$ for which $w^s \neq 0$ is one of $l, 2l$ or $m,$ where $l=1$ or $2$ for $\mathbb{F}=\mathbb{R}$ or $\mathbb{C},$ respectively and $w \in H^{1}(X/G)$ is the  characteristic class of the principle $G$-bundle $G \hookrightarrow X \rightarrow X/G.$ We have index$(X)\leq s$ \cite{floyd}. Thus, we have following Result:
	\begin{theorem}
		Let $G=\mathbb{Z}_2$ act freely on $X \sim_2 \mathbb{F}P^n \times \mathbb{S}^m,$ where $n,m\geq 1.$ Then, there does not exist $G$-equivariant map from $\mathbb{S}^k \rightarrow X,$ for $k > s,$ where $s$ is one of $l,2l$ or $m,$ and $l=1$ or $2$ for $\mathbb{F}=\mathbb{R}$ or  $\mathbb{C},$ respectively.
	\end{theorem}
	Recall that the Volovikov's index $i(X)$ is the smallest integer $r\geq 2$ such that $d_r: E_r^{k-r,r-1} \rightarrow E_r^{k,0}$ is nontrivial for some $k,$ in the Leray-Serre spectral sequence of the Borel fibration  $ X \stackrel{i} \hookrightarrow X_G \stackrel{\pi} \rightarrow B_G.$ By Theorems \ref{thm 4.2} and \ref{thm 4.3}, we get $i(X)$ is  $ l+1 ,2l+1$ or $m+1,$ where $l=1$ or $2$ for $\mathbb{F}=\mathbb{R}$ or $\mathbb{C},$ respectively. By taking $Y=\mathbb{S}^{k}$ in Theorem 1.1 \cite{co}, we have 
	\begin{theorem}
		Let $G=\mathbb{Z}_2$ act freely on a finite CW-complex $X \sim_2 \mathbb{F}P^n \times  \mathbb{S}^m,$ where $ \mathbb{F}=\mathbb{R}$ or  $\mathbb{C}.$ Then, there is no $G$-equivariant map $f: X \rightarrow \mathbb{S}^{k}$  if $1\leq k < i(X)-1,$ where $i(X)=l+1,2l+1$ or $m+1,$ and  $l=1$ or $2$ for $\mathbb{F}=\mathbb{R}$ or $\mathbb{C},$ respectively.
		
	\end{theorem}

	\bibliographystyle{plain}

\end{document}